\documentclass[12pt,reqno]{amsart}
\usepackage{amsmath,amssymb}
\usepackage{cite}

 \makeatletter
\evensidemargin \oddsidemargin
\makeatother

 \newtheorem{theorem}{Theorem}[section]
\newtheorem{lemma}[theorem]{Lemma}
\newtheorem{proposition}[theorem]{Proposition}
\newtheorem{corollary}[theorem]{Corollary}

\theoremstyle{definition}
\newtheorem{definition}[theorem]{Definition}

\theoremstyle{remark}
\newtheorem{remark}[theorem]{Remark}
\numberwithin{equation}{section}

 \newtheorem*{theorem*}{Theorem}

\numberwithin{equation}{section}
\numberwithin{equation}{section}

\begin{document}
\title[ Weighted numerical radii  of accretive matrices]{ Weighted numerical  radii  of accretive matrices }

 \author[A. Sheikhhosseini, M. Khosravi]{ Alemeh
Sheikhhosseini* and Maryam Khosravi}
\address[A. Sheikhhosseini]{Department of
Pure Mathematics, Faculty of Mathematics and Computer,
Shahid Bahonar University of Kerman,
Kerman, Iran}
\email{sheikhhosseini@uk.ac.ir;
hosseini8560@gmail.com}
\address[M. Khosravi]{Department of
Pure Mathematics, Faculty of Mathematics and Computer,
Shahid Bahonar University of Kerman,
Kerman, Iran}
\email{khosravi$_-$m@uk.ac.ir}





 \begin{abstract}\noindent
In this paper, we consider the notion of the weighted numerical radius, denoted by $ \omega_{\nu}(A)$, for the class of accretive matrices, where \\ $ \nu  \in [0, 1].$ This notion generalizes both the classical numerical radius and the operator norm. All of the obtained inequalities reduce to the corresponding classical inequalities when $ \nu=1/2 $ or  when the matrix is positive definite.
\end{abstract}

 \maketitle

 \section{introduction and preliminary}\vspace{.2cm} \noindent

Denote by $\mathbb{M}_n$ the space of all complex $n \times n$ matrices. For a self-adjoint matrix $A \in \mathbb{M}_n$, the notation $A \geq 0$ indicates that $A$ is positive semidefinite, i.e., $\langle Ax, x \rangle \geq 0$ holds for every $x \in \mathbb{C}^n$. When $A$ is positive semidefinite and invertible, we call it positive definite and write $A > 0$.

 A matrix $A\in \mathbb M_n$ is called accretive, if $\Re z$ is positive in its Cartesian (or Toeplitz) decomposition,
$A=\Re z+i\Im z$,
where $\Re z=\frac{A+A^*}{2}$ and
$\Im z=\frac{A-A^*}{2}$. The numerical range of $A\in \mathbb M_n$ is defined by
\begin{equation*}
W(A)=\{x^*Ax:x\in\mathbb C^n,\ \|x\|=1\}.
\end{equation*}
Let $W(A)\subset S_{\alpha}$ for some $0 \leq \alpha <\frac{\pi}{2}$, where
$S_{\alpha}$ denote the sector region in the complex plane as follows:
\begin{equation*}
S_{\alpha}=\{z\in \mathbb C:\Re z>0, \vert\Im z\vert\leq(\Re z)\tan\alpha\}.
\end{equation*}
In this case, we will write $A\in {\mathcal S}_{\alpha}$. Since $0\notin S_{\alpha}$, then each member of $ {\mathcal S}_{\alpha}$ is invertible.

An operator mean $\sigma$ in the sense of Kubo-Ando \cite{a} is defined by an
operator monotone function $f:(0,\infty)\to(0,\infty)$ with $f(1) = 1$ as
$$A\sigma B = A^{1/2}f(A^{-1/2}BA^{-1/2})A^{1/2},$$
for positive invertible operators $A$ and $B$. Here the function $f$ is called the representing function of $\sigma$. Recently, Bedrani et al. proved that this definition can also be used for accretive operators.

Among the most significant operator means are the following:
\begin{itemize}
\item The arithmetic mean: $A \nabla B = \dfrac{A+B}{2}$, and its weighted version:\\  
$A \nabla_t B = tA + (1-t)B$, for $0 < t < 1$.

\item The geometric mean: $A^{1/2} (A^{-1/2} B A^{-1/2})^{1/2} A^{1/2}$, and the weighted\\ geometric mean:  
$A \sharp_{t} B = A^{1/2} (A^{-1/2} B A^{-1/2})^t A^{1/2}$, for $0 < t < 1$.

\item The Heinz mean: $\mathcal{H}_t(A,B) = \dfrac{A \sharp_{t} B + A \sharp_{1-t} B}{2}$, for $0 < t < 1$.

\item The logarithmic mean: $\mathcal{L}(A,B) = \displaystyle\int_0^1 A \sharp_{t} B \, dt$.
\end{itemize}

The function $f$ is said to be matrix monotone if $A \geq B$ with spectra are contained in $J,$ implies $f(A) \geq f(B).$ Also, $f$ is called matrix concave if
$$f(\lambda A+(1-\lambda) B) \geq \lambda f(A)+(1-\lambda) f(B)$$
for all $\lambda \in [0,1]$ and for every Hermitian matrices $A, B \in \mathbb{M}_{n} $ whose spectra are in the interval $J.$\\
Recall that if $f$ is a nonnegative continuous function on $[0, \infty),$ then $f$ is matrix monotone if and only if
$f$ is matrix concave, see \cite[Corollary 1.12]{pec}.\\
We also use the notation
\begin{equation*}
{\textbf{m}}=\lbrace f| f:(0,\infty)\rightarrow(0, \infty)~\text{is a  matrix monotone function with}f(1)=1\rbrace.
\end{equation*}
It is well known that the equation
\begin{equation}\label{m.f}
f(A)\sigma f(B)\leq f(A\sigma B),
\end{equation}
for each $A,B\geq0$, $f\in\textbf{m}$ and operator mean $\sigma$
 
The numerical radius $\omega(A)$ of $A \in \mathbb{M}_n$ is defined by
$$
\omega(A) = \sup \{ |\langle Ax, x \rangle| : x \in \mathbb{C}^n, \, \|x\| = 1 \}.
$$

It is well established that $\omega(\cdot)$ is a norm on $\mathbb{M}_n$, equivalent to the usual operator norm $\|\cdot\|$. Indeed, for every $A \in \mathbb{M}_n$, the following fundamental inequalities hold:

\begin{equation}\label{eq1}
\dfrac{1}{2} \|A\| \leq \omega(A) \leq \|A\|.
\end{equation}
The first inequality in \eqref{eq1} becomes an equality if $A^2=0,$ while the second inequality becomes an equality if $T$ is normal in the sense that $AA^*=A^*A;$ see \cite{gust}.\\
 When $A\in\mathbb{M}_{n}$, the real and imaginary parts of $A$ are defined, respectively, by
$$\Re (A)=A\nabla A^*\;{\text{and}}\;\Im (A)=(-iA)\nabla (-iA)^*,$$ where for two operators $A,B\in\mathbb{M}_{n}$, the arithmetic mean $A\nabla B$ is defined by $A\nabla B=\frac{A+B}{2}.$ Then the Cartesian decomposition of $A$ is $A=\Re (A)+i\Im (A)$.
An important and useful identity for the numerical radius which has been known for researchers is as follows, see \cite{Z-M} for example.
\begin{proposition}\label{prop_yamazaki}
Let $A\in\mathbb{M}_{n}$. Then $$\omega(A)=\sup_{\theta}\|\Re(e^{i\theta}A)\|.$$
\end{proposition}

This identity is a well-known characterization of the numerical radius. and has found various applications. A similar formulation for Hilbert space operators was previously defined by the authors in \cite{sheikh}.

\begin{definition}\cite{sheikh}\label{def_weighted_oper}
Let $A\in\mathbb{M}_{n}$ be any operator and let $0\leq \nu\leq 1.$ We define the weighted real and imaginary parts of $T$ by
$$\Re_{\nu}( A)=A\nabla_{\nu} A^*\;{\text{and}}\;\Im_{\nu} (A)=(-iA)\nabla_{\nu}(-iA)^*,$$ respectively.
\end{definition}

Immediate properties of $\Re_{\nu}$ and $\Im_{\nu}$ are as follows.
\begin{proposition}\cite{sheikh}
Let $A\in\mathbb{M}_{n}$ and let $0\leq \nu\leq 1.$ Then
$$\Im_{\nu}(iA)= \Re_{\nu} (A), ~ \Re_{\nu} (iA) = -\Im_{\nu}(A),$$
\begin{equation}\label{e00}
\Re_{\nu} (A) =\Re (A) +(2\nu-1)i \Im(A),
\end{equation}
\begin{equation}\label{e0}
\Im_{\nu}(A) =\Im(A)- (2\nu-1)i \Re (A)
\end{equation}
and
\begin{equation}\label{eq000}
\Re_{\nu}(A)+i\Im_{\nu}(A)=2\nu A.
\end{equation}
\end{proposition}

\begin{definition}\cite{sheikh}\label{d1}
Let $ 0 \leq \nu \leq 1 $ and $ A \in \mathbb{M}_{n}.$ The $ \nu- $weighted numerical radius of $A$ is denoted by $ \omega_{\nu} (A)$ and is defined by
$$\omega_{\nu}(A)=\sup_{\theta \in \mathbb{R} } \Vert \Re_{\nu} ( e^{i \theta}A) \Vert. $$
Obviously, $ \omega_{\frac{1}{2}}(A)=\omega(A) $ and $ \omega_{0}(A)=\omega_{1}(A) =\Vert A\Vert=\Vert A^*\Vert.$
\end{definition}

 $\omega_\nu(.)$ defines a norm on $\mathbb{M}_{n}$ which is equivalent to the  operator norm \cite{sheikh};
\begin{equation}\label{eq.equvalent}
\max\{\nu,1-\nu\}\|A\|\leq\omega_\nu(A)\leq\|A\|.
\end{equation}


In this paper, we present some inequalities for
$\omega_\nu(\cdot)$ on sectorial matrices, including sharp bounds for powers, 
connections with operator monotone functions, and estimates for operator means. 
These results extend several known inequalities for the numerical radius to the weighted numerical radius of sectorial matrices.


\section{Some Preliminaries}\vspace{.2cm} \noindent
In this section, we recall some auxiliary lemmas that will be used frequently throughout the paper.

\begin{lemma}\label{l1}\cite{sheikh}
Let $ A \in \mathbb{M}_{n}$ and $ 0 \leq \nu \leq 1. $ Then
\begin{equation}\label{e1}
\omega_{\nu}( A)= \omega_{\nu}( A^{*})=\omega_{1- \nu}( A).
\end{equation}
\end{lemma}
\begin{lemma}\cite{bed1}\label{l.sig}
Let $ A, B \in {\mathcal S}_{\alpha}$. Then $A\sigma B \in {\mathcal S}_{\alpha}$ and
\begin{equation*}\label{1.1}
\Re A\sigma \Re B\leq \Re (A\sigma B)\leq\sec^2\alpha(\Re A\sigma \Re B).
\end{equation*}
\end{lemma}

\begin{lemma}\label{l2}\cite{cho}
Let $A,B\in \mathcal{S} _{\alpha}.$ If $t\in[0,1]$, then
\begin{equation}\label{cho-1}
\cos^{2t}(\theta) \Re A^t\leq \Re^t A\leq \Re A^t
\end{equation}
and if $t\in[-1,0],$ then
\begin{equation}\label{cho-2}
\Re A^t \leq \Re^tA\leq\cos^{2t}(\theta) \Re A^t.
\end{equation}
\end{lemma}
 

\begin{lemma}\label{l4}\cite{bed1}
Let $A \in \mathcal{S} _{\alpha}.$ If $ f \in \mathbf{m}, $ then

$$f(\Re A) \leq \Re (f(A)) \leq \sec^{2} \alpha f(\Re A). $$

\end{lemma}

\begin{lemma}\label{l5}\cite{bed1}
Let $A \in \mathcal{S} _{\alpha}.$ If $ f \in \mathbf{m}, $ then

$$f( \Vert \Re A \Vert ) \leq \Vert \Re (f(A)) \Vert \leq \sec^{2} \alpha f(\Vert \Re A \Vert). $$

\end{lemma}

\begin{lemma}\label{l6}\cite{zha3}
Let $A \in \mathcal{S} _{\alpha}.$ Then

$$ \cos\alpha \vert\Vert  A \Vert\vert \leq \vert\Vert \Re (A) \Vert\vert \leq  \vert\Vert  A \Vert\vert. $$

\end{lemma}

\begin{lemma}\label{l7}\cite{Dur2}
Let $A \in \mathcal{S} _{\alpha}$ and $ t \in (0, 1). $ Then 
$ A^{t} \in \mathcal{S} _{t \alpha}. $ Also note that $ A^{-t} \in \mathcal{S} _{t \alpha}, $ by the result indicates that if $ A \in \mathcal{S} _{\alpha} $ then $ A^{-1} \in \mathcal{S} _{\alpha}. $
\end{lemma}
\begin{lemma}\cite{Abu}\label{l8}
Let $ A, B \in \mathbb{M}_{n} $ be positive matrices.Then
\begin{equation*}
\omega \bigg(  \begin{bmatrix}
0& A\\
B & 0
\end{bmatrix}  \bigg)=\frac{1}{2} \Vert  A+B \Vert.
\end{equation*}
\end{lemma}
The following lemma is well known.
\begin{lemma}\label{l9}
Let $ A, B \in \mathbb{M}_{n}. $ Then
\begin{equation*}
 \bigg\Vert  \begin{bmatrix}
0& A\\
B & 0
\end{bmatrix}  \bigg\Vert=\max \lbrace   \Vert  A \Vert, \Vert  B \Vert \rbrace.
\end{equation*}
\end{lemma}

\begin{lemma}\label{l10}
Let $ A, B \in \mathbb{M}_{n}$  be positive. For operator mean $ \sigma $
$$  \Vert\vert A \sigma B \vert\Vert \leq \Vert\vert A \vert\Vert \sigma  \Vert\vert B \vert\Vert. $$
\end{lemma}

\begin{lemma}\label{l11}
Let $A,  B \in \mathcal{S} _{\alpha}$ and $ t \in (0, 1). $ Then 
$$\cos^{3} \alpha \, \Vert\vert A\sharp B \vert\Vert  \leq  \, \Vert\vert \mathrm{H}_{t}(A,  B )\vert\Vert \leq \dfrac{\sec^{3}\alpha}{2}\Vert\vert A+ B \vert\Vert $$
\end{lemma}
\section{main results}\vspace{.2cm} \noindent
We now turn to the main results of this paper.
\begin{lemma}\label{main}
Let $ A \in  \mathbb{M}_{n} $ and $ 0 \leq  \nu \leq 1. $ Then
\begin{equation}\label{l-main-eq.2}
\omega_{\nu} (\Re A) \leq \omega_{\nu} ( A). 
\end{equation} 

\end{lemma}

\begin{proof}
For $ \theta \in \mathbb{R},$ we have
\begin{align*}
2\Vert  \Re_{\nu}(e^{i \theta} \Re A)  \Vert &= 2 \Vert  \nu e^{i \theta} \Re A+ (1-\nu)e^{-i \theta} (\Re A)^{*} \Vert\\ 
&= \Vert  \nu  e^{i \theta}  A+\nu e^{i \theta}  A^{*} + (1-\nu)  e^{-i \theta}  A^{*} + (1-\nu)e^{-i \theta}  A \Vert \\
&\leq   \Vert  \nu e^{i \theta}  A+ (1-\nu)e^{-i \theta}  A^{*} \Vert + \Vert \nu e^{i \theta}  A^{*} + (1-\nu)e^{-i \theta}  A  \Vert \\
&= \Vert \Re_{\nu}(e^{i \theta}  A ) \Vert + \Vert \Re_{\nu}(e^{i \theta}  A^{*})  \Vert \\
&\leq \omega_{\nu} ( A)+ \omega_{\nu} ( A^{*})\\
&=2 \omega_{\nu} ( A).
\end{align*}
Then taking the supremum over $\theta$, with consideration of  Definition \ref{d1} and (\ref{e1}), we obtain the desired result. 
\end{proof}

By Lemma \ref{main}, we have $ \omega_{\nu}(  \Re A) \leq \omega_{\nu}(A) $ for every $ A \in \mathbb M_n. $ A  reversed version can be found for sectorial matrices, as follows.
\begin{proposition}
Let $ A \in \mathcal{S} _{\alpha} $  and $ 0\leq \nu \leq 1.$ Then, 

\begin{equation}\label{eq.2.5}
\omega_{\nu}(A) \leq \sec \alpha \, \omega_{\nu}(  \Re A ).
\end{equation}
\end{proposition}

\begin{proof}
For $ A \in \mathcal{S} _{\alpha} $ since $ \Re A >0, $ then $ \omega_{\nu}( \Re A) = \| \Re A \|. $
So Lemma \ref{l6}, implies
$$    \omega_{\nu}(A) \leq \Vert A \Vert  \leq \sec \alpha \Vert \Re A \Vert =\sec \alpha \omega_{\nu}(  \Re A). $$
\end{proof}

\begin{theorem}\label{t1}
Let $ A \in \mathcal{S} _{\alpha} $  and $ 0\leq \nu \leq 1. $Then
 \begin{equation}\label{eq.2.1}
 \cos\alpha \Vert A\Vert \leq  \omega_{\nu}(A) \leq \Vert A \Vert.
 \end{equation}
In particular,
\begin{equation}\label{eq.2.3}
 \cos\alpha \Vert A\Vert \leq  \omega(A) \leq \Vert A \Vert.
 \end{equation}
 \end{theorem}
\begin{proof}
Note that  since  $ \Re A > 0, $    then $ \omega_{\nu}(\Re A) = \| \Re A \|. $   Also  Lemma \ref{l6} and Lemma \ref{main}, respectively,  implies
$$ \cos\alpha \Vert  A \Vert      \leq \Vert \Re (A) \Vert= \omega_{\nu}(\Re A) \leq \omega_{\nu} (A) \leq \parallel A \parallel.   $$

Choosing $ \nu=\frac{1}{2} $   yields  (\ref{eq.2.3}).
\end{proof}

\begin{remark}
Notice that for $ \frac{1}{2} \leq \nu \leq 1,$ if $ 0 < \alpha < \cos^{-1}(\nu),$ we have $ \cos \alpha > \nu. $ This means for such $ \alpha, $ 

$$ \nu \Vert A \Vert < \cos \alpha \Vert A \Vert  \leq \omega_{\nu}(A) \leq \Vert A \Vert, $$ 
 which presents a significant refinement of the left-hand side of  inequality  (\ref{eq.equvalent}).
\end{remark}
Theorem \ref{t1}, Lemma \ref{l6} and Lemma \ref{l7}, imply the following desired result.
\begin{corollary}\label{cor.1}
Let $ A \in \mathcal{S} _{\alpha} $  and $ 0\leq \nu \leq 1. $Then for $ t \in (-1, 1), $

$$  \cos t \alpha \Vert A^{t}\Vert \leq  \omega_{\nu}(A^{t}) \leq \Vert A^{t} \Vert. $$
\end{corollary}

In the next results, we present weighted sectorial versions of the well known power inequality.

\begin{equation}\label{p.ineq.}
\omega (A^{k}) \leq \omega^{k}(A),  \,  A \in \mathbb{M}_{n}, \, k =1, 2, \ldots
\end{equation}

The following theorem is a generalization of Theorem 3.1 in \cite{bed2}.
\begin{theorem}
Let $ A \in \mathcal{S} _{\alpha} $  and $ 0\leq \nu, t \leq 1. $Then 
\begin{equation}\label{eq.p.nu}
\cos t\alpha \cos^{t}\alpha \, \omega_{\nu}^{t}(A) \leq \omega_{\nu}(A^{t}) \leq \sec t\alpha \sec^{2t}\alpha \, \omega_{\nu}^{t}(A)
\end{equation}
In particular,
\begin{equation}\label{eq.p.0}
\cos t\alpha \cos^{t}\alpha \, \Vert A\Vert^{t} \leq \Vert A^{t} \Vert \leq \sec t\alpha \sec^{2t}\alpha \, \Vert A\Vert^{t},
\end{equation}
 
\end{theorem}

\begin{proof}
Let $ 0\leq \nu, t \leq 1. $ Then
\begin{align*}
\omega_{\nu}(A^{t}) \leq \Vert A^{t}\Vert &\leq \sec t\alpha \Vert \Re A^{t}\Vert \quad (\textit{by Lemma \ref{l6}})\\
&\leq \sec t\alpha \sec^{2t} \alpha \Vert \Re^{t} A\Vert  \quad (\textit{by  (\ref{cho-1})})\\
&= \sec t\alpha \sec^{2t} \alpha \Vert \Re A\Vert^{t} \\
&\leq \sec t\alpha \sec^{2t} \alpha \, \omega_{\nu}^{t}( \Re A)\\
&\leq \sec t\alpha \sec^{2t} \alpha \, \omega_{\nu}^{t}(  A). \quad (\textit{by (\ref{l-main-eq.2})})\\
\end{align*}
Thus, we have proved the second inequality.  To prove the first, we have\\

\begin{align*}
\omega_{\nu}(A^{t}) \geq \cos t\alpha \, \Vert A^{t} \Vert & \geq \cos t\alpha \, \Vert \Re A^{t} \Vert \quad (\textit{by Corollary \ref{cor.1}})\\
& \geq \cos t\alpha \, \Vert \Re^{t} A \Vert \quad (\textit{by  (\ref{cho-1})})\\
& = \cos t\alpha \, \Vert \Re A \Vert^{t} \\
& \geq \cos t\alpha  \cos^{t}\alpha \, \Vert  A \Vert^{t} \quad (\textit{by Lemma \ref{l6}})\\
& \geq \cos t\alpha  \cos^{t}\alpha \, \omega_{\nu}^{t}( A )
\end{align*}
\end{proof}
A   weighted negative-power version of (\ref{eq.p.nu}) can be obtained as follows. 
\begin{theorem}
Let $ A \in \mathcal{S} _{\alpha} $  and $ 0\leq  \nu \leq 1. $Then for $ t \in [-1, 0], $
\begin{equation}\label{eq.p.re}
\cos t\alpha \, \omega_{\nu}^{t}(A)\leq  \cos^{2t}\alpha \, \omega_{\nu}(A^{t}).
\end{equation}

\end{theorem}
\begin{proof}
For $ t \in [-1, 0], $ we have
\begin{align*}
\omega_{\nu}(A^{t}) \geq \cos t\alpha \, \Vert A^{t} \Vert & \geq \cos t\alpha \, \Vert \Re A^{t} \Vert \quad (\textit{by Corollary \ref{cor.1}})\\
& \geq \cos t\alpha \, \cos^{-2t} \alpha \Vert \Re^{t} A \Vert \quad (\textit{by  (\ref{cho-2})})\\
& = \cos t\alpha \, \cos^{-2t} \alpha \Vert \Re A \Vert^{t} \\
& \geq \cos t\alpha  \cos^{-2t} \alpha \, \omega_{\nu}^{t} (\Re A )\\
& \geq \cos t\alpha  \cos^{-2t} \alpha\, \omega_{\nu}^{t}( A )\quad (\textit{by ( \ref{l-main-eq.2})})
\end{align*}
\end{proof}
 Choosing $ \nu=0 $  in (\ref{eq.p.re}), we have  the following inequalities.
 \begin{corollary}
 Let $ A \in \mathcal{S} _{\alpha}. $ Then for $ t \in [-1, 0], $
 \begin{equation}\label{eq.p.re.1}
\cos t\alpha \, \Vert A\Vert^{t} \leq  \cos^{2t}\alpha \, \Vert A^{t}\Vert,
\end{equation}

In particular if $ A $ is positive,
\begin{equation*}
\Vert A\Vert^{t} \leq   \Vert A^{t}\Vert.
\end{equation*}
 \end{corollary}
 
 \begin{proposition}
 Let $ A, B \in \mathcal{S} _{\alpha}. $ Then for $ \nu \in [0, 1], $
 \begin{equation}\label{eq.pro.1}
 \omega_{\nu}(AB) \leq \sec^{2} \alpha\, \omega_{\nu}(A) \omega_{\nu}(B). 
 \end{equation}
 In particular,
 \begin{equation}\label{eq.pro.2}
 \omega(AB) \leq \sec^{2} \alpha\, \omega(A) \omega(B), 
 \end{equation}
 \end{proposition}
 \begin{proof}
 Let $ \nu \in [0, 1]. $ Then
 \begin{align*}
 \omega_{\nu}(AB) &\leq \Vert A B\Vert\\
 &\leq \Vert A \Vert  \Vert  B\Vert\\
 &\leq \sec^{2}\alpha  \Vert \Re A \Vert  \Vert \Re B\Vert  \quad (\textit{by Lemma \ref{l6}})\\
 &=\sec^{2}\alpha \, \omega_{\nu}(\Re A) \omega_{\nu}(\Re B)\\
 &\leq \sec^{2}\alpha \, \omega_{\nu}( A) \omega_{\nu}( B),   \quad (\textit{by( \ref{l-main-eq.2})})\\
 \end{align*}
 which completes the proof of (\ref{eq.pro.1}). Let $ \nu=\frac{1}{2}, $  in (\ref{eq.pro.1}),  inequality (\ref{eq.pro.2}) is deduced.
 \end{proof}
\section{ The weighted numerical radius and operator monotone functions}
The aim of this section is to derive some inequalities involving $\omega_{\nu}(f(A))$ and operator monotone functions.
\begin{lemma}\cite{bed2}\label{l.op.mean}
Let $ A, B \in \mathcal{S} _{\alpha}, \, f \in \textbf{m}$ and $\sigma$ be an operator mean.  Then
$ f(A) \in \mathcal{S} _{\alpha}$ and $ A \sigma B \in \mathcal{S}_{\alpha}. $
\end{lemma}

\begin{theorem}
Let $ A \in \mathcal{S} _{\alpha},  0 \leq \nu \leq 1 $ and $ f \in \textbf{m}. $ Then
\begin{equation}\label{eq.f-w}
\cos \alpha f(\omega_{\nu}(A)) \leq \omega_{\nu}(f(A)) \leq \sec^{3} \alpha f(\omega_{\nu}(A)),
\end{equation}
and
\begin{equation}\label{eq. f-w-1}
\cos \alpha f(\Vert A \Vert ) \leq \Vert f(A) \Vert \leq \sec^{3} \alpha f(\Vert A \Vert).
\end{equation}

\end{theorem}
\begin{proof}
Recall that every nonnegative operator monotone function   $ f $ and $ 0  \leq t \leq 1, $ satisfies
$ t f(x) \leq f(tx). $ Let $ 0 \leq \nu \leq 1. $ Then
\begin{align*}
\cos \alpha f(\omega_{\nu}(A))   & \leq    f( \cos \alpha \omega_{\nu}(A))   \quad (\textit{by (\ref{eq.2.5})})\\
& \leq f( \omega_{\nu}( \Re A))\\
&= f( \Vert  \Re A  \Vert )\\
& \leq \Vert \Re  f(A)   \Vert   \quad (\textit{by Lemma \ref{l5}})\\
&= \omega_{\nu}( \Re f(A))   \leq \omega_{\nu}( f(A)).
\end{align*}
This completes the proof of the first inequality. To proof the second inequality, we have
	\begin{align*}
\omega_{\nu}(f(A)) \leq \Vert f(A) \Vert & \leq \sec \alpha  \Vert \Re f(A) \Vert  \quad (\textit{by Lemma \ref{l6}})\\
&\leq \sec^{3}\alpha \, f(\Vert \Re A \Vert)   \quad (\textit{by Lemma \ref{l5}})\\
&= \sec^{3}\alpha \, f(\omega_{\nu}( \Re A )) \\
&\leq \sec^{3}\alpha \, f(\omega_{\nu}(  A )). \\
\end{align*} 
If we put $\nu=0$ or 1 in \eqref{eq.f-w}, we get \eqref{eq. f-w-1}.
\end{proof}
\begin{remark}
Note that if  $ A $ is positive, then $ \alpha =0, $  and we have $ \omega_{\nu}(f(A)) = f(\omega_{\nu}(  A )).$
In particular, we obtain $ \omega(f(A)) = f(\omega(  A )) $ and $ \Vert f(A)\Vert  = f(\Vert  A \Vert) $.
\end{remark}

\begin{theorem}\label{pro. con}
Let $ A, B \in \mathcal{S} _{\alpha} $ and $ f \in \textbf{m}. $ Then for $ \lambda, \nu \in [0,1], $

\begin{equation}\label{eq. con-nu}
\omega_{\nu}(f(A)\nabla_{\lambda} f(B)) \leq \sec^{3} \alpha f(\omega_{\nu}(A)\nabla_{\lambda} \omega_{\nu}(B)).
\end{equation}
In particular,
\begin{equation*}
\|f(A)\nabla_{\lambda} f(B)\| \leq \sec^{3} \alpha f(\|A\|)\nabla_{\lambda} \|B\|).
\end{equation*}
\end{theorem}
\begin{proof}
Since $ \omega_{\nu} $ is a norm and $ f $ is concave, we have 
\begin{align*}
\omega_{\nu}(\lambda f(A)+(1-\lambda) f(B)) & \leq \lambda \omega_{\nu}(f(A))+(1-\lambda) \omega_{\nu}(f(B))  \\
& \leq \sec^{3}\alpha \left(  \lambda f( \omega_{\nu}(A))+(1-\lambda) f(\omega_{\nu}(B))   \right)  \quad \textit{by (\ref{eq.f-w})}\\
&\leq \sec^{3}\alpha f( \lambda\omega_{\nu}(A)+(1-\lambda) \, \omega_{\nu}(B)).  \\
\end{align*}
The proof is completed.
\end{proof} 

In Theorem \ref{pro. con}, let $ f(x)=x^{t}, t \in [0, 1], $ and $ \lambda=\frac{1}{2}, $
we get the following inequality.\\

\begin{corollary}\label{cor.01}
Let $ A, B \in \mathcal{S} _{\alpha}. $ Then for $  \nu, t \in [0,1], $
\begin{equation}\label{meq1} \omega_{\nu} (A^{t}+B^{t}) \leq 2^{1-t} \sec^{3} \alpha \, (\omega_{\nu}(A)+\omega_{\nu}(B))^{t}.
\end{equation}
\end{corollary}
If  in last inequality, we set $ \nu=1/2 $ and $ \nu=0, $  respectively, we have\\
$$  \omega (A^{t}+B^{t}) \leq 2^{1-t} \sec^{3} \alpha \, (\omega(A)+\omega(B))^{t} $$

and
$$  \Vert A^{t}+B^{t}\Vert \leq 2^{1-t} \sec^{3} \alpha \, \left( \Vert A\Vert +\Vert B \Vert \right)^{t}. $$

Another version of Theorem \ref{pro. con} can be stated as follows.
\begin{theorem}\label{m1}
Let $ A, B \in \mathcal{S} _{\alpha} $ and $ f \in \textbf{m}. $ Then for $ \lambda, \nu \in [0,1], $
\begin{equation*}
\omega_{\nu}(f(A)\nabla_{\lambda} f(B)) \leq \sec^{3} \alpha\ \omega_{\nu}(f(A\nabla_{\lambda} B)).
\end{equation*}
\end{theorem}
\begin{proof}
For $ A, B \in \mathcal{S} _{\alpha} $,
\begin{align*}
\omega_{\nu}( f(A)\nabla_{\lambda} f(B)) & \leq \|f(A)\nabla_{\lambda} f(B)\|\\
&\leq \sec\alpha \|\Re(f(A)\nabla_{\lambda} f(B))\|\quad (\textit{by Lemma \ref{l6}})\\
&= \sec\alpha \|\Re(f(A))\nabla_{\lambda}\Re(f(B))\|\\
&\leq \sec^3\alpha \|f(\Re A)\nabla_{\lambda} f(\Re B)\|\quad (\textit{by Lemma \ref{l4}})\\
&\leq \sec^3\alpha \|f(\Re A\nabla_{\lambda} \Re B)\|\quad (\textit{by concavity of $f$})\\
&= \sec^3\alpha \|f(\Re (A\nabla_{\lambda} B))\|\\
&\leq \sec^3\alpha \|\Re(f (A\nabla_{\lambda} B))\|\quad (\textit{by Lemma \ref{l4}})\\
&\leq \sec^3\alpha\ \omega_{\nu}(f (A\nabla_{\lambda} B))\quad (\textit{by Lemma \ref{main}}).
\end{align*}
\end{proof}
As an immediate consequence of this theorem, we get
\begin{equation}\label{meq2}\omega_{\nu}(A^t+B^t)\leq2^{1-t}\sec^3 \alpha\ \omega_{\nu}(A+B)^t,\end{equation}
for each $0\leq t\leq 1$.

Since $\|.\|$ and $\omega$ are special case of $\omega_{\nu}$, these inequalities are also hold for operator norm and numerical radius.
\begin{remark}
If either $\nu=0$ or $A$ and $B$ are normal operators, $\omega_{\nu}$ is equal to operator norm and it is eaily sean that 
$$\omega_{\nu}(A+B)^t\leq(\omega_{\nu}(A)+\omega_{\nu}(B))^t,$$
for $0\leq t\leq1$. So inequality \eqref{meq2} presents a better bound than inequality \eqref{meq1}. However, it may not be true in general case. 

\end{remark}
It is known that (see \cite{An2}) for any unitary invariant norm $ \Vert\vert . \vert\Vert $ on $ \mathbb{M}_{n}, $ 
\begin{equation}\label{eq.sub.u}
\Vert\vert f(A+B)  \vert\Vert  \leq  \Vert\vert f(A)+f(B)  \vert\Vert, \, \, A, B \geq 0, \, f \in \textbf{m}. 
\end{equation} 

We now state this inequality for the weighted numerical radius for sectorial matrices.
  
\begin{theorem}
Let $ A, B \in \mathcal{S} _{\alpha}. $ If   $ f \in \textbf{m}$ and $  \nu \in [0,1], $ then
\begin{equation}\label{eq.sub.nu}
\omega_{\nu}\big(f(A+B)\big) \leq \sec^{3}\alpha \, \omega_{\nu}\big(f(A)+f(B)\big).
\end{equation}
\end{theorem}
\begin{proof}
Let $ 0 \leq \nu \leq 1. $ Then we have
\begin{align*}
\omega_{\nu}\big(f(A+B)\big) \leq \Vert f(A+B)  \Vert & \leq \sec\alpha \,  \Vert \Re f(A+B)  \Vert \quad (\textit{by Lemma \ref{l6}})\\
 & \leq \sec^{3}\alpha  \, \Vert  f(\Re A + \Re B)  \Vert \quad (\textit{by Lemma \ref{l4}})\\
 & \leq \sec^{3}\alpha \, \Vert  f(\Re A) + f( \Re B)  \Vert \quad (\textit{by \eqref{eq.sub.u}})\\
 & \leq \sec^{3}\alpha \, \Vert \Re \big( f( A) + f(  B) \big) \Vert \quad (\textit{by Lemma \ref{l4}})\\
 &=\sec^{3}\alpha \,  \omega_{\nu}\big(  \Re \big( f( A) + f(  B) \big)   \big)\\
 &\leq \sec^{3}\alpha \, \omega_{\nu}\big(   f( A) + f(  B)   \big)\quad (\textit{by Lemma \ref{main}}),\\
\end{align*}
which completes the proof.
\end{proof}

Putting $ f(x)=x^{t} $ for $ t \in [0, 1]$  in (\ref{eq.sub.nu}), we have the following reverse of inequality \eqref{meq2}.
\begin{corollary}\label{cor.3}
Let $ A, B \in \mathcal{S} _{\alpha}. $ If   $ 0 \leq t, \nu \leq 1,$  then
\begin{equation}\label{eq.sub.nu.3}
\omega_{\nu}\big((A+B)^{t}) \leq \sec^{3}\alpha \, \omega_{\nu}\big(A^{t}+B^{t}\big).
\end{equation}
\end{corollary}

When $ A, B \in \mathbb{M}_{n} $ are  positive and $ 0 \leq \nu \leq 1, $ then it  is clear that\\ $ \omega_{\nu}(A+B) \geq \max \lbrace \omega_{\nu}(A), \omega_{\nu}(B) \rbrace. $
If either $ A $ or $ B $ is not positive, this inequality is not necessary  true. So, if $ A, B $
are sectorial, we have the following version.
\begin{theorem}
Let $ A, B \in \mathcal{S} _{\alpha}. $ If   $ 0 \leq  \nu \leq 1,$  then

\begin{equation}\label{eq.sum}
\cos^{2}\alpha \,\max \lbrace \omega_{\nu}(A), \omega_{\nu}(B) \rbrace \leq \omega_{\nu}(A+B).
\end{equation}
\end{theorem}
\begin{proof}
Let $ A, B \in \mathcal{S} _{\alpha}$ and   $ 0 \leq  \nu \leq 1.$  Then
\begin{align*}
\omega_{\nu}(A+B) &\geq \cos \alpha \, \Vert  A+B \Vert \quad \textit{(by   (\ref{eq.2.1}))}\\
& \geq \cos \alpha \, \Vert  \Re A+\Re B \Vert \quad \textit{(by  Lemma \ref{l6})}\\
&=2\cos \alpha \,  \omega \bigg( \begin{bmatrix}
0 & \Re A\\
\Re B & 0
\end{bmatrix} \bigg) \quad \textit{(by  Lemma \ref{l8})}\\
& \geq \cos \alpha \, \bigg\Vert   \begin{bmatrix}
0 & \Re A\\
\Re B & 0
\end{bmatrix} \bigg\Vert \quad \textit{(by  (\ref{eq1}))}\\
&=\cos \alpha \, \max \big\lbrace  \Vert \Re A \Vert, \Vert \Re B \Vert \big\rbrace \quad \textit{(by  Lemma \ref{l9})} \\
&=\cos \alpha \, \max \big\lbrace  \omega_{\nu} (\Re A ), \omega_{\nu} (\Re B ) \big\rbrace  \\
&\geq \cos^{2} \alpha \, \max \big\lbrace  \omega_{\nu} ( A ), \omega_{\nu} ( B ) \big\rbrace  \quad \textit{(by  (\ref{eq.2.5}))}.\\
\end{align*}
This completes the proof.
\end{proof}
Let $ \nu=0 $  or $ \nu=1 $  in (\ref{eq.sum}), we have the following inequality.
\begin{corollary}
Let $ A, B \in \mathcal{S} _{\alpha}. $ Then
\begin{equation}\label{eq.sum.2}
\cos^{2}\alpha \,\max \big\lbrace \Vert A\Vert, \Vert B \Vert \big\rbrace \leq \Vert A+B \Vert.
\end{equation}
\end{corollary}
\section{The weighted numerical radius and operator means}\vspace{.2cm} \noindent
This section is devoted to establishing inequalities connecting \(\omega_{\nu}(A\sigma B)\) and \(\omega_{\nu}(A)\sigma\omega_{\nu}(B)\) for sectorial matrices.

\begin{theorem}
Let $ A, B \in \mathcal{S} _{\alpha} $ and $ 0 \leq \nu \leq 1. $ For each operator mean $ \sigma $,
\begin{equation*}
\omega_{\nu}(A \sigma B) \leq \sec^{3}\alpha \, \big( \omega_{\nu}(A) \sigma \omega_{\nu}(B)  \big).
\end{equation*}
In particular,
\begin{equation*}
\Vert A \sigma B \Vert \leq \sec^{3}\alpha \, \big( \Vert A \Vert \sigma \Vert B\Vert \big).
\end{equation*}
\end{theorem}
\begin{proof}
In view of Lemma \ref{l.op.mean}, we have
\begin{align*}
\omega_{\nu}(A \sigma B) &\leq \Vert A \sigma B  \Vert\\
&\leq \sec \alpha\, \Vert  \Re (A \sigma B)  \Vert \quad \textit{(by  Lemma \ref{l6})}\\
&\leq \sec^{3} \alpha\, \Vert ( \Re A) \sigma  (\Re B)  \Vert \quad \textit{(by  Lemma \ref{l.sig})}\\
&\leq \sec^{3} \alpha\,  \big(\Vert  \Re A \Vert  \sigma  \Vert \Re B  \Vert \big) \quad \textit{(by  Lemma \ref{l10})}\\
&= \sec^{3} \alpha\,  \big(\omega_{\nu} ( \Re A)   \sigma  \omega_{\nu} (\Re B)   \big) \\
&\leq \sec^{3} \alpha\,  \big(\omega_{\nu} (  A)   \sigma  \omega_{\nu} ( B)   \big)  \quad \textit{(by  Lemma \ref{main})}.\\
\end{align*}
The special case is obtained  by letting $\nu=0$.
\end{proof}
This theorem extends Theorem 3.7 of \cite{bed2} and Theorem 8.2 of \cite{bed1}.

Choosing different operator means $\sigma$ yields the following relations for $ A, B \in \mathcal{S} _{\alpha} $ and $ 0 \leq \nu \leq 1$
\begin{align*}
\omega_{\nu}(A \sharp_{t} B) \leq \sec^{3}\alpha \, \big( \omega_{\nu}^{1-t} (A)  \omega_{\nu}^{t} (B)  \big),\\
\omega_{\nu}(\mathcal{L}(A, B))\leq \sec^{3}\alpha \, \mathcal{L}(\omega_{\nu}(A),\omega_{\nu}( B)),\\
\omega_{\nu}(\mathcal{H}_{t}(A, B))\leq \sec^{3}\alpha \, \mathcal{H}_{t}(\omega_{\nu}(A),\omega_{\nu}( B)).
\end{align*}

In addition, we have the following result.
 \begin{theorem}
Let $ A, B \in \mathcal{S} _{\alpha} $ and $ 0 \leq \nu \leq 1. $ For each $f\in\textbf{m}$ and operator mean $ \sigma $,
\begin{equation}
\omega_{\nu}(f(A) \sigma f(B)) \leq \sec^{5}\alpha \,  \omega_{\nu}(f(A \sigma B)) .
\end{equation}
In particular,
\begin{equation*}
\Vert f(A) \sigma f(B) \Vert \leq \sec^{5}\alpha \, \Vert f(A  \sigma B)\Vert.
\end{equation*}
\end{theorem}
\begin{proof}
Similar to the proof of Theorem \ref{m1},
\begin{align*}
\omega_{\nu}( f(A)\sigma f(B)) & \leq \|f(A)\sigma f(B)\|\\
&\leq \sec\alpha \|\Re(f(A)\sigma f(B))\|\quad (\textit{by Lemma \ref{l6}})\\
&\leq \sec^3\alpha \|\Re(f(A))\sigma\Re(f(B))\|\quad (\textit{by Lemma \ref{l.sig}})\\
&\leq \sec^5\alpha \|f(\Re A)\sigma f(\Re B)\|\quad (\textit{by Lemma \ref{l4}})\\
&\leq \sec^5\alpha \|f(\Re A\sigma \Re B)\|\quad (\textit{by \eqref{m.f}})\\
&\leq \sec^5\alpha \|f(\Re (A\sigma B))\|\quad (\textit{by Lemma \ref{l.sig}})\\
&\leq \sec^5\alpha \|\Re(f (A\sigma B))\|\quad (\textit{by Lemma \ref{l4}})\\
&\leq \sec^5\alpha\ \omega_{\nu}(f (A\sigma B))\quad (\textit{by Lemma \ref{main}}).
\end{align*}
The special case is obtained  by letting $\nu=0$.
\end{proof}

The following theorem states a Heinz-type inequality for the weighted numerical radius of accretive matrices.
\begin{theorem}
Let $ A, B \in \mathcal{S} _{\alpha} $ and $ 0 \leq  t, \nu \leq 1. $ Then
\begin{equation}\label{eq.H}
\cos^{4}\alpha \,  \omega_{\nu}(A\sharp B)  \leq \omega_{\nu}( \mathcal{H}_{t}(A, B))\leq \dfrac{\sec^{4}\alpha}{2}\omega_{\nu}(A+ B),
\end{equation}
and
\begin{equation*}
\cos^{4}\alpha \, \Vert A\sharp B\Vert \leq \Vert \mathcal{H}_{t}(A, B)\Vert \leq \dfrac{\sec^{4}\alpha}{2}\Vert A+ B\Vert.
\end{equation*}

\end{theorem}
\begin{proof}
For the first inequality,
\begin{align*}
\omega_{\nu}(A\sharp B)  &\leq    \Vert A\sharp B \Vert\\
& \leq \sec^{3}\alpha \, \Vert  \mathcal{H}_{t}(A, B)\Vert  \quad \textit{(by  Lemma \ref{l11})}\\
& \leq \sec^{4}\alpha \, \omega_{\nu} (\mathcal{H}_{t}(A, B)). \quad \textit{(by (\ref{eq.2.1}))} \\
\end{align*}
Also for the second inequality
\begin{align*}
\omega_{\nu}( \mathcal{H}_{t}(A, B)) &\leq \Vert \mathcal{H}_{t}(A, B) \Vert\\
&\leq \sec^{3}\alpha \, \bigg\Vert \dfrac{A+B}{2} \bigg\Vert \quad \textit{(by  Lemma \ref{l11})}\\
&\leq   \dfrac{\sec^{4}\alpha }{2}\omega_{\nu}(A+B).  \quad \textit{(by (\ref{eq.2.1}))} \\
\end{align*}
which completes the proof.
\end{proof}


\end{document}